\documentclass[12pt]{article}
\usepackage{graphicx}
\usepackage[ruled,vlined]{algorithm2e}
\usepackage{appendix}
\usepackage{amsmath}
\usepackage{amssymb}
\usepackage{amsthm}
\usepackage{multirow}
\usepackage{color}
\usepackage{longtable}
\usepackage{array}
\usepackage{url}
\usepackage{booktabs}
\usepackage{float}
\usepackage{mathtools}
\usepackage{tikz}
\usepackage{caption}

\newtheorem{theorem}{Theorem}[section]

\newtheorem{lemma}[theorem]{Lemma}

\newtheorem{corollary}[theorem]{Corollary}

\theoremstyle{definition}

\newcommand{\turan}{Tur\'{a}n }
\newcommand{\erdos}{Erd\H{o}s}

\title{New lower bounds for the Tur\'{a}n density of $PG_{m}(q)$}
\author{Tao Zhang$^{\text{a,}}$\thanks{e-mail: tzh@zju.edu.cn. Research supported by the National Natural Science Foundation of China under Grant No. 11801109.}~ and  Gennian Ge$^{\text{b,}}$\thanks{e-mail: gnge@zju.edu.cn. Research supported by the National Natural Science Foundation of China under Grant No. 11971325, National Key Research and Development Program of China under Grant  No. 2018YFA0704703,  and Beijing Scholars Program.}\\
\footnotesize $^{\text{a}}$ School of Mathematics and Information Science, Guangzhou University, Guangzhou 510006, China. \\
\footnotesize $^{\text{b}}$ School of Mathematics Sciences, Capital Normal University, Beijing 100048, China.\\}

\begin{document}

\date{}

\maketitle

\begin{abstract}
Let $\mathcal{H}$ be an $r$-uniform hypergraph. The \turan number $\text{ex}(n,\mathcal{H})$ is the maximum number of edges in an $n$-vertex $\mathcal{H}$-free $r$-uniform hypergraph. The \turan density of $\mathcal{H}$ is defined by
\[\pi(\mathcal{H})=\lim_{n\rightarrow\infty}\frac{\text{ex}(n,\mathcal{H})}{\binom{n}{r}}.\]
In this paper, we consider the \turan density of projective geometries. We give two new constructions of $PG_{m}(q)$-free hypergraphs which improve some results given by Keevash (J. Combin. Theory Ser. A, 111: 289--309, 2005). Based on an upper bound of  blocking sets of $PG_m(q)$,
 we give a new general lower bound for the \turan density of $PG_{m}(q)$. By a detailed analysis of the structures of complete arcs in $PG_2(q)$, we also get better lower bounds for the \turan density of $PG_2(q)$ with $q=3,\ 4,\ 5,\ 7,\ 8$.

\medskip


\smallskip

\end{abstract}

\section{Introduction}
Let $\mathcal{H}$ be an $r$-uniform hypergraph. An $r$-uniform hypergraph $\mathcal{G}$ is called $\mathcal{H}$-free if $\mathcal{G}$ contains no copy of $\mathcal{H}$ as its subhypergraph. The \emph{\turan number} $\text{ex}(n,\mathcal{H})$ is the maximum number of edges in an $n$-vertex $\mathcal{H}$-free $r$-uniform hypergraph.

The study of \turan numbers is one of the central problems in extremal combinatorics. For simple graphs ($r=2$), the \erdos-Stone theorem \cite{ES1946} completely determined the asymptotics of $\text{ex}(n,H)$ when $H$ is not a bipartite graph. For bipartite graphs, the determination of the exact asymptotics of $\text{ex}(n,H)$ is far from being solved.

In contrast to the simple graph case, there are only a few results for the hypergraph \turan problems. For example, even the asymptotic value of $\text{ex}(n,K_{t}^{(r)})$ is still unknown for any $t> r>2$, where $K_{t}^{(r)}$ denotes the complete $r$-uniform hypergraph on $t$ vertices. For an $r$-uniform hypergraph, let
\[\pi(\mathcal{H})=\lim_{n\rightarrow\infty}\frac{\text{ex}(n,\mathcal{H})}{\binom{n}{r}}.\]
It is not hard to show that $\pi(\mathcal{H})$ exists for any $r$-uniform hypergraph $\mathcal{H}$, and it is usually called the \emph{\turan density}. It is conjectured that $\pi(K_{4}^{(3)})$ is equal to $5/9$. Recently, there have been some new results for the hypergraph \turan problems (for example, see \cite{FPS2003,FPS2005,KeeS2005,MR2002}). For more extremal results, we refer the readers to the surveys \cite{F1991,K11}.

In this paper, we focus on the \turan problem of $PG_{m}(q)$. For the case $m=q=2$, which is also known as the Fano plane, the \turan density is $\pi(PG_{2}(2))=\frac{3}{4}$ \cite{DF2000}. Later, the exact \turan number was determined for sufficiently large $n$: $\text{ex}(n,PG_{2}(2))=\binom{n}{3}-\binom{\lfloor \frac{n}{2}\rfloor}{3}-\binom{\lceil\frac{n}{2}\rceil}{3}$. This result was proved simultaneously and independently by F\"{u}redi and Simonovits \cite{FS2005} and Keevash and Sudakov \cite{KS2005}. In \cite{K2005}, Keevash systematically studied the \turan density of projective geometries. He proved the following general bound for $\pi(PG_{m}(q))$.
\begin{theorem}\label{pgthmbefore}\rm{\cite{K2005}}
The \turan density of $PG_{m}(q)$ satisfies
\[\prod_{i=1}^{q}\Big(1-\frac{i}{\sum_{j=1}^{m}q^{j}}\Big)\le \pi(PG_{m}(q))\le 1-\frac{1}{\binom{q^{m}}{q}}.\]
\end{theorem}
For $q=2$, he also improved the above upper bound:
 \[\pi(PG_{m}(2))\le\begin{cases}1-\frac{3}{2^{2m}-1},&\textup{ if } m \text{ is odd},\\
1-\frac{6}{(2^{m}-1)(2^{m+1}+1)},&\textup{ if } m \text{ is even}.\end{cases}\]

For the case $q=2$ and $m=3$, Cioab\v{a} \cite{C2004} proved the bound $\frac{27}{32}\le\pi(PG_{3}(2))\le\frac{27}{28}$, and this was improved to
\[3\sqrt{3}+2\sqrt{2(9-5\sqrt{3})}-6\le\pi(PG_{3}(2))\le\frac{13}{14}\]
 by Keevash \cite{K2005}.

In \cite{K2005}, Keevash also gave a lower bound of $\pi(PG_{m}(q))$ via chromatic number:
\[\pi(PG_{m}(q))\ge1-\frac{1}{(\chi(PG_{m}(q))-1)^{q}},\]
where $\chi(PG_{m}(q))$ denotes the chromatic number of $PG_{m}(q)$.
 Clearly, $PG_{m}(q)$ has chromatic number 2 if and only if it has a blocking set. In \cite{T1988}, Tallini showed that there is a blocking set in $PG_{3}(q)$ if and only if $q\ge5$. Hence $\chi(PG_{3}(3))\ge3$ and $\chi(PG_{3}(4))\ge3$, then we have
\[\pi(PG_{3}(3))\ge\frac{7}{8},\ \pi(PG_{3}(4))\ge\frac{15}{16}.\]

The lower bound and upper bound in Theorem~\ref{pgthmbefore} are quite close when $m$ is sufficiently large. When $m$ is small, we have seen that the lower bound can be improved for $m=2$, $q=2$; $m=3$, $q=2$; $m=3$, $q=3$ and $m=3$, $q=4$.
In this paper, our goal is to continue this investigation. We will give more cases of $PG_{m}(q)$ for which the lower bounds of \turan density are better than those of Theorem~\ref{pgthmbefore}. Our first result is on the base of an upper bound of  blocking
sets in $PG_{m}(q)$.
\begin{theorem}\label{pgthm1}
If the maximal size of  blocking sets in $PG_{m}(q)$ is $k$, then the \turan density of $PG_{m}(q)$ satisfies
\[\pi(PG_{m}(q))\ge\max\Big\{(q+1)!\sum_{i=1}^{q}\binom{\sum_{i=1}^{m}q^{i}-k}{q+1-i}\frac{(1-(\sum_{i=1}^{m}q^{i}-k)\alpha)^{i}}{i!}\alpha^{q-i+1}:\ 0\le\alpha\le\frac{1}{\sum_{i=1}^{m}q^{i}-k}\Big\}.\]
\end{theorem}

By an upper bound of  blocking sets in $PG_{m}(q)$  (Lemma~\ref{pglemm1}), we have the following corollary.
\begin{corollary}\label{pgcoro1}
Let $t=\lceil\sum_{i=0}^{m-2}q^{i}(q+\sqrt{q})\rceil$. Then the \turan density of $PG_{m}(q)$ with $q\ge5$ satisfies
\[\pi(PG_{m}(q))\ge\max\Big\{(q+1)!\sum_{i=1}^{q}\binom{t}{q+1-i}\frac{(1-t\alpha)^{i}}{i!}\alpha^{q-i+1}:\ 0\le\alpha\le\frac{1}{t}\Big\}.\]
And for $m=2$ and $q\ge2$, we have
\[\pi(PG_{2}(q))\ge\max\Big\{(q+1)!\sum_{i=1}^{q}\binom{q+\lceil\sqrt{q}\rceil}{q+1-i}\frac{(1-(q+\lceil\sqrt{q}\rceil)\alpha)^{i}}{i!}\alpha^{q-i+1}:\ 0\le\alpha\le\frac{1}{q+\lceil\sqrt{q}\rceil}\Big\}.\]
\end{corollary}

In Tables~\ref{comparisontable1} and \ref{comparisontable2}, we give a comparison between the lower bounds from Theorem~\ref{pgthmbefore} and the lower bounds from Corollary~\ref{pgcoro1} for $m=2,3$ and small $q$.
\begin{table}[h]
\begin{center}
\caption{Comparisons for $\pi(PG_{2}(q))$}
\begin{tabular}{|p{2cm}|p{3.5cm}|p{3.5cm}|p{4.5cm}|}
\hline
&Lower bound from Theorem~\ref{pgthmbefore}  &Lower bound from Corollary~\ref{pgcoro1}&Corresponding $\alpha$ taken in Corollary~\ref{pgcoro1}\\ \hline
$\pi(PG_{2}(3))$&$\sim0.5729$  & $\sim\bf{0.69586}$&$\sim0.0809$\\ \hline
$\pi(PG_{2}(4))$&$\sim0.5814$  & $\sim\bf{0.70699}$&$\sim0.0576$\\ \hline
$\pi(PG_{2}(5))$&$\sim0.5864$  & $\sim\bf{0.7347}$&$\sim0.0389$\\ \hline
$\pi(PG_{2}(7))$&$\sim0.59218$  & $\sim\bf{0.7480}$&$\sim0.0247$\\ \hline
$\pi(PG_{2}(8))$&$\sim0.59397$  & $\sim\bf{0.7548}$&$\sim0.0205$\\ \hline
$\pi(PG_{2}(9))$&$\sim0.59536$  & $\sim\bf{0.7614}$&$\sim0.0173$\\ \hline
$\pi(PG_{2}(11))$&$\sim0.597389$  & $\sim\bf{0.78166}$&$\sim0.0122$\\ \hline
$\pi(PG_{2}(13))$&$\sim0.59879$  & $\sim\bf{0.7914}$&$\sim0.0095$\\ \hline
$\pi(PG_{2}(16))$&$\sim0.6002$  & $\sim\bf{0.8043}$&$\sim0.0069$\\ \hline
$\pi(PG_{2}(17))$&$\sim0.6006$  & $\sim\bf{0.8130}$&$\sim0.0061$\\ \hline
$\pi(PG_{2}(19))$&$\sim0.6012$  & $\sim\bf{0.8197}$&$\sim0.0051$\\ \hline
\end{tabular}
\label{comparisontable1}
\end{center}
\end{table}

\begin{table}[h]
\begin{center}
\caption{Comparisons for $\pi(PG_{3}(q))$}
\begin{tabular}{|p{2cm}|p{3.5cm}|p{3.5cm}|p{4.5cm}|}
\hline
&Lower bound from Theorem~\ref{pgthmbefore}  &Lower bound from Corollary~\ref{pgcoro1}&Corresponding $\alpha$ taken in Corollary~\ref{pgcoro1}\\ \hline
$\pi(PG_{3}(17))$&$\sim\bf{0.9710777103}$  & $\sim0.9701091221$&$\sim0.0006198906$\\ \hline
$\pi(PG_{3}(19))$&$\sim\bf{0.9740717446}$  & $\sim0.9736668015$&$\sim0.0004716926$\\ \hline
$\pi(PG_{3}(23))$&$\sim0.9785208385$  & $\sim\bf{0.9790232680}$&$\sim0.0002926917$\\ \hline
$\pi(PG_{3}(25))$&$\sim0.9802185562$  & $\sim\bf{0.9809821553}$&$\sim0.0002383002$\\ \hline
$\pi(PG_{3}(27))$&$\sim0.9816677623$  & $\sim\bf{0.9827113542}$&$\sim0.0001961485$\\ \hline
$\pi(PG_{3}(29))$&$\sim0.9829192657$  & $\sim\bf{0.9841874880}$&$\sim0.0001636689$\\ \hline
\end{tabular}
\label{comparisontable2}
\end{center}
\end{table}

Our second construction is based on the property of complete arcs in $PG_{2}(q)$.
Let $K$ be an arc of $PG_{2}(q)$, and $m(K)$ be the minimal size of sets of points $Q$ such that the set $PG_{2}(q)\backslash(K\cup Q)$ contains no line. Let
\[M(q)=\min\{m(K):\ K \text{ is a complete arc in }PG_{2}(q)\}.\]
It is easy to see that $M(q)\le q-1$. Then we have the following result.
\begin{theorem}\label{pgthm7}
The \turan density of $PG_{2}(q)$ satisfies
\begin{align*}
\pi(PG_{2}(q))\ge&\max\Big\{(q+1)!\sum_{i=1}^{q}\sum_{j=\max\{0,q+2-M(q)-i\}}^{\min\{2,q+1-i\}}\frac{1}{i!j!}\binom{M(q)-1}{q+1-i-j}\alpha^{i}\beta^{j}\gamma^{q+1-i-j}:\\ &\alpha,\beta,\gamma\ge0,\alpha+\beta+(M(q)-1)\gamma=1\Big\}.
\end{align*}
In particular, we have
\begin{enumerate}
  \item[(1)] $\pi(PG_{2}(3))\ge0.7364719055,$
  \item[(2)] $\pi(PG_{2}(4))\ge0.7381611274,$
  \item[(3)] $\pi(PG_{2}(5))\ge0.7440388117,$
  \item[(4)] $\pi(PG_{2}(7))\ge0.7583661147,$
  \item[(5)] $\pi(PG_{2}(8))\ge0.7654160822.$
\end{enumerate}
\end{theorem}

The rest of this paper is organized as follows. In Section~\ref{pre}, we give some basics of projective geometries. In Section~\ref{section3}, we prove Theorem \ref{pgthm1}. Section~\ref{secssr} contains the proof of Theorem~\ref{pgthm7}. Section~\ref{conclusion} concludes the paper.
\section{Preliminaries}\label{pre}
Let $\mathbb{F}_{q}$ denote the finite field with $q$ elements. The projective geometry of dimension $m$ over $\mathbb{F}_{q}$, denoted by $PG_{m}(q)$, is the following $(q+1)$-graph. Its vertex set is the point set of $PG_{m}(q)$, that is the set of all one-dimensional subspaces of $\mathbb{F}_{q}^{m+1}$. Its edges are the lines of $PG_{m}(q)$, i.e., the two-dimensional subspaces of $\mathbb{F}_{q}^{m+1}$, in which for each two-dimensional subspace, the set of one-dimensional subspaces that it contains is an edge of the hypergraph $PG_{m}(q)$.

A \emph{blocking set} $\mathcal{B}$ in $PG_{m}(q)$ is a subset of $PG_{m}(q)$ which meets every line but contains no line completely; that is $1\le|\mathcal{B}\cap \ell|\le q$ for every line $\ell$ in $PG_{m}(q)$.
The following bound for blocking sets in $PG_{m}(q)$ can be found in \cite{T1995}.

\begin{lemma}\label{pglemm1}\rm{\cite[P. 322-323]{T1995}}
 If $\mathcal{B}$ is a blocking set in $PG_{m}(q)$ with $q\ge5$, then
\[|\mathcal{B}|\le q^{n}-\sqrt{q}(q^{n-2}+q^{n-3}+\dots+1).\]
Moreover, if $\mathcal{B}$ is a blocking set in $PG_{2}(q)$ with $q>2$, then
\[|\mathcal{B}|\le q^{2}-\sqrt{q}.\]
\end{lemma}

A \emph{$k$-arc} in $PG_{2}(q)$ is a set $K$ of $k$ points for which no three are collinear. A $k$-arc is \emph{complete} if it is not contained in a $(k+1)$-arc.
\begin{lemma}\label{pglemm2}\rm{\cite[Theorem 8.5]{H2005}}
Let $K$ be a $k$-arc in $PG_{2}(q)$, if $q$ is odd then $k\le q+1$. If $q$ is even then $k\le q+2$.
\end{lemma}

 An \emph{oval} of  $PG_{2}(q)$ is a $(q+1)$-arc, and a $(q+2)$-arc of $PG_{2}(q)$ with $q$ even is called a \emph{hyperoval}.  A line $\ell$ is called \emph{$i$-secant} of $K$ if $|\ell\cap K|=i$. Sometimes, a $0$-secant is called \emph{a passant}, a $1$-secant is called \emph{a tangent} and a $2$-secant is called \emph{a secant}. Then we have the following lemmas.
\begin{lemma}\label{pglemm10}\rm{\cite[Section 14.2]{H2005}}
The complete arc in $PG_{2}(3)$ must be a $4$-arc, which is an oval. For any oval of $PG_{2}(3)$, there are $3$ passants and they are not concurrent.
\end{lemma}
\begin{lemma}\label{pglemm3}\rm{\cite[Section 14.3]{H2005}}
The complete arc in $PG_{2}(4)$ must be a $6$-arc, which is a hyperoval. For any hyperoval of $PG_{2}(4)$, there are $6$ passants and no three of them are concurrent.
\end{lemma}
\begin{lemma}\label{pglemm4}\rm{\cite[Section 14.4]{H2005}}
The complete arc in $PG_{2}(5)$ must be a $6$-arc, which is an oval.
For any oval of $PG_{2}(5)$, there are $10$ passants and no four of them are concurrent.
\end{lemma}
\begin{lemma}\label{pglemm7}\rm{\cite[Section 8.1, Section 9.1, Section 9.3]{H2005}}
The size of a complete arc in $PG_{2}(7)$ is $6$ or $8$.
For any $8$-arc of $PG_{2}(7)$, there are $21$ passants and no five of them are concurrent.
\end{lemma}
\begin{lemma}\label{pglemm8}\rm{\cite[Section 8.1, Section 9.1, Section 9.3]{H2005}}
The size of a complete arc in $PG_{2}(8)$ is $6$ or $10$.
For any $10$-arc of $PG_{2}(8)$, there are $28$ passants and no five of them are concurrent.
\end{lemma}

\section{General result}\label{section3}
Our first construction depends on the upper bound of blocking sets in $PG_{m}(q)$.
\begin{proof}[Proof of Theorem~\ref{pgthm1}]
We partition a set $V$ of $n$ vertices into parts $X,\ Y_{1},\ Y_{2},\dots, Y_{\sum_{i=1}^{m}q^{i}-k}$ so that $||X|-\beta n|\le1$ and $||Y_{i}|-\alpha n|\le1$, where $\alpha,\beta$ are positive constants with $\beta+(\sum_{i=1}^{m}q^{i}-k)\alpha=1$. Define a $(q+1)$-uniform hypergraph $\mathcal{H}_{n}$ on vertex set $V$ such that the edges of $\mathcal{H}_{n}$ are all $(q+1)$ tuples $e$ of $V$ satisfying $1\le|e\cap X|\le q$ and $|e\cap Y_{i}|\le1$ for $i=1,\dots,\sum_{i=1}^{m}q^{i}-k$.

Assume that there exists a $PG_{m}(q)$ in $\mathcal{H}_{n}$, and let $A=X\cap PG_{m}(q)$ and $B_{i}=Y_{i}\cap PG_{m}(q)$ ($i=1,\dots,\sum_{i=1}^{m}q^{i}-k$). Then $A$ is a blocking set, hence $|A|\le k$. Since for any two vertices of $PG_{m}(q)$, there is an edge through them, then $|B_{i}|\le1$ for $i=1,\dots,\sum_{i=1}^{m}q^{i}-k$. Hence \[\sum_{i=0}^{m}q^{i}=|PG_{m}(q)|=|A|+\sum_{i=1}^{\sum_{i=1}^{m}q^{i}-k}|B_{i}|\le \sum_{i=1}^{m}q^{i},\] which is a contradiction.

Now we count the number of edges in $\mathcal{H}_{n}$.
\begin{align*}
e(\mathcal{H}_{n})&\ge\sum_{i=1}^{q}\binom{\lfloor\beta n\rfloor}{i}\binom{\sum_{i=1}^{m}q^{i}-k}{q-i+1}(\lfloor\alpha n\rfloor)^{q-i+1}\\
&=\sum_{i=1}^{q}\binom{\sum_{i=1}^{m}q^{i}-k}{q-i+1}\frac{\beta^{i}}{i!}\alpha^{q-i+1}n^{q+1}+o(n^{q+1}).
\end{align*}
Therefore, we have a lower bound
\begin{align*}
\pi(PG_{m}(q))&\ge\lim_{n\rightarrow\infty}\binom{n}{q+1}^{-1}e(\mathcal{H}_{n})\\
&=(q+1)!\sum_{i=1}^{q}\binom{\sum_{i=1}^{m}q^{i}-k}{q-i+1}\frac{\beta^{i}}{i!}\alpha^{q-i+1}\\
              &=(q+1)!\sum_{i=1}^{q}\binom{\sum_{i=1}^{m}q^{i}-k}{q-i+1}\frac{(1-(\sum_{i=1}^{m}q^{i}-k)\alpha)^{i}}{i!}\alpha^{q-i+1}.
\end{align*}
\end{proof}
\section{Dimension two: Proof of Theorem~\ref{pgthm7}}\label{secssr}
In this section, by the behavior of passants of the complete arcs in $PG_{2}(q)$, we can give better lower bounds for $\pi(PG_{2}(q))$ when $q=3,4,5,7$ and $8$.

We partition a set $V$ of $n$ vertices into parts $X, Y, Z_{1},Z_{2},\dots,Z_{M(q)-1}$, so that $||X|-\alpha n|\le1$, $||Y|-\beta n|\le1$ and $||Z_{i}|-\gamma n|\le1$ ($i=1,2,\dots,M(q)-1$), where $\alpha,\beta,\gamma\ge0$ and $\alpha+\beta+(M(q)-1)\gamma=1$. Define a $(q+1)$-uniform hypergraph $\mathcal{H}_{n}$ on vertex set $V$ such that the edges of $\mathcal{H}_{n}$ are all $(q+1)$ tuples $e$ of $V$ satisfying $1\le|e\cap X|\le q$, $|e\cap Y|\le2$ and $|e\cap Z_{i}|\le1$.

Assume that there exists a $PG_{2}(q)$ in $\mathcal{H}_{n}$, and let $A=X\cap PG_{2}(q)$, $B=Y\cap PG_{2}(q)$ and $C_{i}=Z_{i}\cap PG_{2}(q)$. Note that $1\le|e\cap A|\le q$ for any edge $e\in PG_{2}(q)$, then $A$ is a blocking set. Since no three vertices of $B$ are collinear, then $B$ is an arc. Since for any two vertices of $PG_{2}(q)$, there is an edge through them, then $|C_{i}|\le1$. Note that $M(q)$ is the minimal number $t$ such that for any complete arc $K$, there is a set $Q$ with $|Q|=t$ and the set $PG_{2}(q)\backslash(K\cup Q)$ contains no line.
Then $PG_{2}(q)\backslash (B\cup_{i=1}^{M(q)-1} C_{i})$ contains a line. Since $A= PG_{2}(q)\backslash (B\cup_{i=1}^{M(q)-1} C_{i})$, then we get a contradiction.

Now we count the number of edges in $\mathcal{H}_{n}$.
\begin{align*}
e(\mathcal{H}_{n})\ge\sum_{i=1}^{q}\sum_{j=\max\{0,q+2-M(q)-i\}}^{\min\{2,q+1-i\}}\binom{M(q)-1}{q+1-i-j}\binom{\lfloor\alpha n\rfloor}{i}\binom{\lfloor\beta n\rfloor}{j}(\lfloor\gamma n\rfloor)^{q+1-i-j}.
\end{align*}
Therefore, we have a lower bound
\begin{align*}
\pi(PG_{2}(q))\ge&\max\Big\{(q+1)!\sum_{i=1}^{q}\sum_{j=\max\{0,q+2-M(q)-i\}}^{\min\{2,q+1-i\}}\frac{1}{i!j!}\binom{M(q)-1}{q+1-i-j}\alpha^{i}\beta^{j}\gamma^{q+1-i-j}:\\ &\alpha,\beta,\gamma\ge0,\alpha+\beta+(M(q)-1)\gamma=1\Big\}.
\end{align*}
In the following subsections, we give better lower bounds for the \turan density of $PG_{2}(3)$, $PG_{2}(4)$, $PG_{2}(5)$, $PG_{2}(7)$ and $PG_{2}(8)$.
\subsection{$PG_{2}(3)$}
By Lemma~\ref{pglemm10}, the complete arc in $PG_{2}(3)$ must be a $4$-arc. For any $4$-arc $K$ of $PG_{2}(3)$, there are $3$ passants and they are not concurrent, hence $M(3)=2$.
Therefore, we have a lower bound
\begin{align*}
\pi(PG_{2}(3))\ge4\alpha^{3}\beta+4\alpha^{3}\gamma+6\alpha^{2}\beta^{2}+12\alpha^{2}\beta\gamma+12\alpha\beta^{2}\gamma,
\end{align*}
where $\alpha,\beta,\gamma\ge0$ and $\alpha+\beta+\gamma=1$.
This lower bound is optimised by the following choice of parameters:
\begin{align*}
&\alpha\sim0.5948588940,\\
&\beta\sim0.3216013121,\\
&\gamma\sim0.0835397939.
\end{align*}
This gives the lower bound
\[\pi(PG_{2}(3))\ge0.7364719055\]
as required.

\subsection{$PG_{2}(4)$}
By Lemma~\ref{pglemm3}, the complete arc in $PG_{2}(4)$ must be a $6$-arc. For any $6$-arc $K$ of $PG_{2}(4)$, there are $6$ passants and no three of them are concurrent. Hence $M(4)=3$.
Therefore, we have a lower bound
\begin{align*}
\pi(PG_{2}(4)) \ge5\alpha^{4}\beta+10\alpha^{4}\gamma+10\alpha^{3}\beta^{2}+40\alpha^{3}\beta\gamma+20\alpha^{3}\gamma^{2}+60\alpha^{2}\beta^{2}\gamma+60\alpha^{2}\beta\gamma^{2}+60\alpha\beta^{2}\gamma^{2},
\end{align*}
where $\alpha,\beta,\gamma\ge0,\alpha+\beta+2\gamma=1$.
This lower bound is optimised by the following choice of parameters:
\begin{align*}
&\alpha\sim0.6566212797,\\
&\beta\sim0.2297814643,\\
&\gamma\sim0.1135972558.
\end{align*}
This gives the lower bound
\[\pi(PG_{2}(4))\ge0.7381611274\]
as required.

\subsection{$PG_{2}(5)$}
By Lemma~\ref{pglemm4}, the complete arc in $PG_{2}(5)$ must be a $6$-arc.
For any $6$-arc $K$ of $PG_{2}(5)$, there are $10$ passants and no four of them are concurrent. Hence $M(5)=4$.
Therefore, we have a lower bound
\begin{align*}
\pi(PG_{2}(5)) \ge&6\alpha^{5}\beta+18\alpha^{5}\gamma+15\alpha^{4}\beta^{2}+90\alpha^{4}\beta\gamma+90\alpha^{4}\gamma^{2}+180\alpha^{3}\beta^{2}\gamma+360\alpha^{3}\beta\gamma^{2}+120\alpha^{3}\gamma^{3}\\
&+540\alpha^{2}\beta^{2}\gamma^{2}+360\alpha^{2}\beta\gamma^{3}+360\alpha\beta^{2}\gamma^{3},
\end{align*}
where $\alpha,\beta,\gamma\ge0,\ \alpha+\beta+3\gamma=1$.
This lower bound is optimised by the following choice of parameters:
\begin{align*}
&\alpha\sim0.7000841083,\\
&\beta\sim0.1750121987,\\
&\gamma\sim0.0416345643.
\end{align*}
This gives the lower bound
\[\pi(PG_{2}(5))\ge0.7440388117\]
as required.

\subsection{$PG_{2}(7)$}
By Lemma~\ref{pglemm7}, the size of a complete arc in $PG_{2}(7)$ is $6$ or $8$.
For any $8$-arc of $PG_{2}(7)$, there are $21$ passants and no five of them are concurrent. For the complete $6$-arc of $PG_{2}(7)$, we have the following lemma.
\begin{lemma}\label{pglemm6}
For any complete $6$-arc $K$ of $PG_{2}(7)$, there do not exist five points $P_{1},P_{2},P_{3},P_{4},\\ P_{5}\in PG_{2}(7)$ such that $PG_{2}(7)\backslash(K\cup\{P_{1},P_{2},P_{3},P_{4},P_{5}\})$ contains no line.
\end{lemma}
The proof of this lemma is just a tedious calculation, which will be given in Appendix A.

Then we have $M(7)=6$.
Therefore, we have a lower bound
\begin{align*}
\pi(PG_{2}(7))\ge&20160\alpha\beta^{2}\gamma^{5}+50400\alpha^{2}\beta^{2}\gamma^{4}+20160\alpha^{2}\beta\gamma^{5}+33600\alpha^{3}\beta^{2}\gamma^{3}+33600\alpha^{3}\beta\gamma^{4}\\&
+6720\alpha^{3}\gamma^{5}+8400\alpha^{4}\beta^{2}\gamma^{2}+16800\alpha^{4}\beta\gamma^{3}+8400\alpha^{4}\gamma^{4}+840\alpha^{5}\beta^{2}\gamma+3360\alpha^{5}\beta\gamma^{2}\\
&+3360\alpha^{5}\gamma^{3}+28\alpha^{6}\beta^{2}+280\alpha^{6}\beta\gamma+560\alpha^{6}\gamma^{2}+8\alpha^{7}\beta+40\alpha^{7}\gamma,
\end{align*}
where $\alpha,\beta,\gamma\ge0,\alpha+\beta+5\gamma=1$.
This lower bound is optimised by the following choice of parameters:
\begin{align*}
&\alpha\sim0.7578927975,\\
&\beta\sim0.1142680556,\\
&\gamma\sim0.02556782938.
\end{align*}
This gives the lower bound
\[\pi(PG_{2}(7))\ge0.7583661147\]
as required.

\subsection{$PG_{2}(8)$}
By Lemma~\ref{pglemm8}, the size of a complete arc in $PG_{2}(8)$ is $6$ or $10$.
For any $10$-arc of $PG_{2}(8)$, there are $28$ passants and no five of them are concurrent. For the complete $6$-arc of $PG_{2}(8)$, we have the following lemma.
\begin{lemma}\label{pglemm9}
For any complete $6$-arc $K$ of $PG_{2}(8)$, there do not exist six points $P_{1},P_{2},P_{3},P_{4},\\ P_{5},P_{6}\in PG_{2}(8)$ such that $PG_{2}(8)\backslash(K\cup\{P_{1},P_{2},P_{3},P_{4},P_{5},P_{6}\})$ contains no line.
\end{lemma}
The proof of this lemma is given in Appendix B.

Then we have $M(8)=7$.
Therefore, we have a lower bound
\begin{align*}
\pi(PG_{2}(8))\ge&181440\alpha\beta^2\gamma^6+544320\alpha^2\beta^2\gamma^5+181440\alpha^2\beta\gamma^6+453600\alpha^3\beta^2\gamma^4+362880\alpha^3\beta\gamma^5\\
&+60480\alpha^3\gamma^6+151200\alpha^4\beta^2\gamma^3+226800\alpha^4\beta\gamma^4+90720\alpha^4\gamma^5+22680\alpha^5\beta^2\gamma^2\\
&+60480\alpha^5\beta\gamma^3+45360\alpha^5\gamma^4+1512\alpha^6\beta^2\gamma+7560\alpha^6\beta\gamma^2+10080\alpha^6\gamma^3+36\alpha^7\beta^2\\
&+432\alpha^7\beta\gamma+1080\alpha^7\gamma^2+9\alpha^8\beta+54\alpha^8\gamma,
\end{align*}
where $\alpha,\beta,\gamma\ge0,\alpha+\beta+6\gamma=1$.
This lower bound is optimised by the following choice of parameters:
\begin{align*}
&\alpha\sim0.7782735564,\\
&\beta\sim0.0960589824,\\
&\gamma\sim0.0209445768.
\end{align*}
This gives the lower bound
\[\pi(PG_{2}(8))\ge0.7654160822\]
as required.

\section{Conclusion}\label{conclusion}
In this paper, we give some improvements on the lower bounds of the \turan density of $PG_{m}(q)$. We have the following remarks.
\begin{itemize}
  \item The construction for the lower bound in Theorem \ref{pgthm1} depends on the upper bound of blocking sets of $PG_{m}(q)$, any improvements may give better bounds for the \turan density of $PG_{m}(q)$. It is easy to see that the lower bound of $\pi(PG_{m}(q))$ from Theorem \ref{pgthm1} is less than $1-\frac{1}{2^{q}}$, while Theorem~\ref{pgthmbefore} gives $1-O(q^{2-m})$ for large $m$, hence it is worse than Theorem~\ref{pgthmbefore} when $m$ is relatively large. But when $q$ is relatively large, from Tables~\ref{comparisontable1} and \ref{comparisontable2}, it seems that our construction is better than Theorem~\ref{pgthmbefore}.
  \item Theorem~\ref{pgthm7} gives better lower bounds which are based on the value of $M(q)$. Recall that $M(q)$ is the minimal number $t$ such that for any complete arc $K$, there is a set $Q$ with $|Q|=t$ and the set $PG_{2}(q)\backslash(K\cup Q)$ contains no line. When $q\ge9$, there are many kinds of complete arcs in $PG_{2}(q)$ other than ovals or hyperovals (see \cite[Section 9.3]{H2005}), hence the determination of $M(q)$ becomes much more complicated. It would be interesting to determine $M(q)$ for general prime power $q$.
  \item A set $K$ of points of $PG_{m}(q)$ is said to be a set of class $[m_{1},\dots,m_{k}]_{r}$, $1\le r\le n-1$, if for every $r$-dimension subspace $\pi$, $|\pi\cap K|=m_{i}$ for some $1\le i\le k$. In this paper, our constructions mainly use the following kinds of sets:
      \begin{enumerate}
        \item a set of class $[0,1]_{1}$ in $PG_{m}(q)$;
        \item a set of class $[0,1,2]_{1}$ in $PG_{2}(q)$;
        \item a set of class $[1,2,\dots,q]_{1}$ in $PG_{m}(q)$.
      \end{enumerate}
      Applying other kinds of sets may give better constructions.
\end{itemize}
\section*{Appendix A: Proof of Lemma~\ref{pglemm6}}
In this section, we label each point of $PG_{2}(7)$ as $(x,y,z)$, i.e., the 1-dimensional subspace spanned by a nonzero vector $(x,y,z)$. Let $[x,y,z]$ denote the homogeneous coordinates of the line of $PG_{2}(7)$, i.e., the $2$-dimension subspace. Here, a point $(a,b,c)$ lies on the line $[ x,y,z]$ if $ax+by+cz=0$.

By \cite[Section 14.5]{H2005}, there are exactly 2 complete $6$-arcs up to isomorphisms:
\begin{align*}
&K_{1}=\{(-1,1,1),(1,1,1),(1,-1,1),(-1,-1,1),(0,2,1),(0,3,1)\},\\
&K_{2}=\{(-1,1,1),(1,1,1),(1,-1,1),(-1,-1,1),(0,2,1),(0,-3,1)\}.
\end{align*}

{\bf{Case 1: If the complete arc is $K_{1}$.}}

It is easy to compute that there are $24$ passants of $K_{1}$. By \cite[Section 9.1]{H2005}, there is no point of $PG_{2}(7)\backslash K_{1}$ through which more than $5$ passants of $K_{1}$ pass. Let
\begin{align*}
&P_{1}=(0,1,0),\ P_{2}=(1,2,6),\ P_{3}=(1,6,5),\\
&P_{4}=(1,5,1),\ P_{5}=(0,0,1),\ P_{6}=(1,1,2),
\end{align*}
and
\[S(P)=\{\ell: \ell\text{ is a passant of } K_{1}, P\in\ell\}.\]
Then we can compute to get that $P_{i}$ ($i=1,2,3,4,5,6$) are all the points of $PG_{2}(7)\backslash K_{1}$ such that there are exactly $5$ passants of $K_{1}$ passing through, and
\begin{align*}
&S(P_{1})=\{    [ 1, 0, 4 ], [ 0, 0, 1 ],[ 1, 0, 3 ],[ 1, 0, 2 ],[ 1, 0, 5 ]\},\\
&S(P_{2})=\{    [ 1, 2, 5 ],[ 1, 3, 0 ],[ 0, 1, 2 ],[ 1, 6, 6 ],[ 1, 1, 3 ]\},\\
&S(P_{3})=\{[ 1, 0, 4 ],[ 1, 5, 5 ],[ 0, 1, 3 ],[ 1, 6, 1 ],[ 1, 3, 6 ]\},\\
&S(P_{4})=\{[ 1, 1, 1 ],[ 0, 1, 2 ],[ 1, 6, 4 ],[ 1, 5, 2 ],[ 1, 4, 0 ]\},\\
&S(P_{5})=\{[ 1, 3, 0 ],[ 1, 2, 0 ],[ 0, 1, 0 ],[ 1, 5, 0 ],[ 1, 4, 0 ]\},\\
&S(P_{6})=\{[ 0, 1, 3 ],[ 1, 0, 3 ],[ 1, 1, 6 ],[ 1, 2, 2 ],[ 1, 4, 1 ]\}.
\end{align*}

Since
\begin{align*}
&S(P_{1})\cap S(P_{3})=\{[1,0,4]\},\ S(P_{1})\cap S(P_{6})=\{[1,0,3]\},\ S(P_{2})\cap S(P_{4})=\{[0,1,2]\},\\
&S(P_{2})\cap S(P_{5})=\{[1,3,0]\},\ S(P_{3})\cap S(P_{6})=\{[0,1,3]\},\ S(P_{4})\cap S(P_{5})=\{[1,4,0]\},
\end{align*}
then for any $I\subset[1,6]$ with $|I|=4$, we have $|\cup_{i\in I}S(P_{i})|\le 19$; and for any $I\subset[1,6]$ with $|I|=5$, we have $|\cup_{i\in I}S(P_{i})|\le 23$. Hence there do not exist five points $P_{1},P_{2},P_{3},P_{4}, P_{5}\in PG_{2}(7)$ such that $PG_{2}(7)\backslash(K_{1}\cup\{P_{1},P_{2},P_{3},P_{4},P_{5}\})$ contains no line.

{\bf{Case 2: If the complete arc is $K_{2}$.}}

It is easy to compute that there are $24$ passants of $K_{2}$. By \cite[Section 9.1]{H2005}, there is no point of $PG_{2}(7)\backslash K_{2}$ through which more than $5$ passants of $K_{2}$ pass. Let
\begin{align*}
&P_{1}=(1,3,3),\ P_{2}=(0,1,0),\ P_{3}=(1,2,5),\\
&P_{4}=(1,5,2),\ P_{5}=(0,0,1),\ P_{6}=(1,4,4),
\end{align*}
and
\[S(P)=\{\ell: \ell\text{ is a passant of } K_{2}, P\in\ell\}.\]
Then we can compute to get that $P_{i}$ ($i=1,2,3,4,5,6$) are all the points of $PG_{2}(7)\backslash K_{2}$ such that there are exactly $5$ passants of $K_{2}$ passing through, and
\begin{align*}
&S(P_{1})=\{[ 1, 1, 1 ],[ 1, 6, 3 ],[ 1, 2, 0 ],[ 1, 0, 2 ],[ 1, 3, 6 ]\},\\
&S(P_{2})=\{ [ 1, 0, 4 ],[ 0, 0, 1 ],[ 1, 0, 3 ],[ 1, 0, 2 ],[ 1, 0, 5 ]\},\\
&S(P_{3})=\{[ 1, 0, 4 ],[ 1, 3, 0 ],[ 1, 6, 3 ],[ 1, 5, 2 ],[ 1, 4, 1 ]\},\\
&S(P_{4})=\{ [ 1, 2, 5 ],[ 1, 0, 3 ],[ 1, 1, 4 ],[ 1, 4, 0 ],[ 1, 3, 6 ]\},\\
&S(P_{5})=\{[ 1, 3, 0 ],[ 1, 2, 0 ],[ 0, 1, 0 ],[ 1, 5, 0 ],[ 1, 4, 0 ]\},\\
&S(P_{6})=\{[ 1, 6, 6 ],[ 1, 5, 0 ],[ 1, 1, 4 ],[ 1, 4, 1 ],[ 1, 0, 5 ]\}.
\end{align*}

Since
\begin{align*}
&S(P_{1})\cap S(P_{2})=\{[1,0,2]\},\ S(P_{1})\cap S(P_{3})=\{[1,6,3]\},\ S(P_{1})\cap S(P_{4})=\{[1,3,6]\},\\
&S(P_{1})\cap S(P_{5})=\{[1,2,0]\},\ S(P_{2})\cap S(P_{3})=\{[1,0,4]\},\ S(P_{2})\cap S(P_{4})=\{[1,0,3]\},\\
&S(P_{2})\cap S(P_{6})=\{[1,0,5]\},\ S(P_{3})\cap S(P_{5})=\{[1,3,0]\},\ S(P_{3})\cap S(P_{6})=\{[1,4,1]\},\\
&S(P_{4})\cap S(P_{5})=\{[1,4,0]\},\ S(P_{4})\cap S(P_{6})=\{[1,1,4]\},\ S(P_{5})\cap S(P_{6})=\{[1,5,0]\},
\end{align*}
then for any $I\subset[1,6]$ with $|I|=4$, we have $|\cup_{i\in I}S(P_{i})|\le 19$; and for any $I\subset[1,6]$ with $|I|=5$, we have $|\cup_{i\in I}S(P_{i})|\le 23$. Hence there do not exist five points $P_{1},P_{2},P_{3},P_{4}, P_{5}\in PG_{2}(7)$ such that $PG_{2}(7)\backslash(K_{2}\cup\{P_{1},P_{2},P_{3},P_{4},P_{5}\})$ contains no line.

\section*{Appendix B: Proof of Lemma~\ref{pglemm9}}
In this section, we label each point of $PG_{2}(8)$ as $(x,y,z)$, i.e., the 1-dimensional subspace spanned by a nonzero vector $(x,y,z)$. Let $[x,y,z]$ denote the homogeneous coordinates of the line of $PG_{2}(8)$, i.e., the $2$-dimension subspace. Here, a point $(a,b,c)$ lies on the line $[ x,y,z]$ if $ax+by+cz=0$.

Let $\omega$ be the primitive element of $\mathbb{F}_{8}$ with $\omega^3+\omega^{2}+1=0$.
By \cite[Section 14.6]{H2005}, there is exactly 1 complete $6$-arc up to isomorphisms: \[K=\{(1,0,0),(0,1,0),(0,0,1),(1,1,1),(\omega^3,\omega^2,1),(\omega^2,\omega^3,1)\}.\]

It is easy to compute that there are $34$ passants of $K$. By \cite[Section 9.1]{H2005}, there is no point of $PG_{2}(8)\backslash K$ through which more than $6$ passants of $K$ pass. Let
\begin{align*}
&P_{1}=(1,\omega^6,0),\ P_{2}=(1,1,0),\ P_{3}=(1,1,\omega^5),\ P_{4}=(1,1,\omega^4),\\
&P_{5}=(1,\omega,0),\ P_{6}=(1,0,1),\ P_{7}=(0,1,1),
\end{align*}
and
\[S(P)=\{\ell: \ell\text{ is a passant of } K, P\in\ell\}.\]
Then we can compute to get that $P_{i}$ ($i=1,2,3,4,5,6,7$) are all the points of $PG_{2}(8)\backslash K$ such that there are exactly $6$ passants of $K$ passing through, and
\begin{align*}
&S(P_{1})=\{    [ 1, \omega, \omega^4 ],
    [ 1, \omega, \omega^3 ],
    [ 1, \omega, \omega^2 ],
    [ 1, \omega, \omega ],
    [ 1, \omega, 1 ],
    [ 1, \omega, \omega^6 ]\},\\
&S(P_{2})=\{[ 1, 1, \omega ],
    [ 1, 1, \omega^6 ],
    [ 1, 1, \omega^5 ],
    [ 1, 1, \omega^4 ],
    [ 1, 1, \omega^3 ],
    [ 1, 1, \omega^2 ]\},\\
&S(P_{3})=\{[ 1, \omega^5, \omega^3 ],
    [ 1, \omega^3, \omega^4 ],
    [ 1, \omega, 1 ],
    [ 1, \omega^6, \omega^6 ],
    [ 1, \omega^4, \omega ],
    [ 1, \omega^2, \omega^5 ]\},\\
&S(P_{4})=\{[ 1, \omega^6, 1 ],
    [ 1, \omega^3, \omega^5 ],
    [ 1, \omega, \omega ],
    [ 1, \omega^4, \omega^2 ],
    [ 1, \omega^5, \omega^4 ],
    [ 1, \omega^2, \omega^6 ]\},\\
&S(P_{5})=\{ [ 1, \omega^6, 1 ],
    [ 1, \omega^6, \omega^3 ],
    [ 1, \omega^6, \omega^2 ],
    [ 1, \omega^6, \omega^5 ],
    [ 1, \omega^6, \omega^6 ],
    [ 1, \omega^6, \omega]\},\\
&S(P_{6})=\{ [ 1, \omega^6, 1 ],
    [ 1, \omega^2, 1 ],
    [ 1, \omega^5, 1 ],
    [ 1, \omega, 1 ],
    [ 1, \omega^3, 1 ],
    [ 1, \omega^4, 1 ]\},\\
&S(P_{7})=\{ [ 1, \omega^2, \omega^2 ],
    [ 1, \omega^4, \omega^4 ],
    [ 1, \omega, \omega ],
    [ 1, \omega^6, \omega^6 ],
    [ 1, \omega^5, \omega^5 ],
    [ 1, \omega^3, \omega^3 ]\}.
\end{align*}

Since
\begin{align*}
&S(P_{1})\cap S(P_{3})=S(P_{1})\cap S(P_{6})=S(P_{3})\cap S(P_{6})=\{[1,\omega,1]\},\\
&S(P_{1})\cap S(P_{4})=S(P_{1})\cap S(P_{7})=S(P_{4})\cap S(P_{7})=\{[1,\omega,\omega]\},\\
&S(P_{3})\cap S(P_{5})=S(P_{3})\cap S(P_{7})=S(P_{5})\cap S(P_{7})=\{[1,\omega^6,\omega^6]\},\\
&S(P_{4})\cap S(P_{5})=S(P_{4})\cap S(P_{6})=S(P_{5})\cap S(P_{6})=\{[1,\omega^6,1]\},
\end{align*}
then for any $I\subset[1,7]$ with $|I|=4$, we have $|\cup_{i\in I}S(P_{i})|\le 23$; for any $I\subset[1,7]$ with $|I|=5$, we have $|\cup_{i\in I}S(P_{i})|\le 28$; and for any $I\subset[1,7]$ with $|I|=6$, we have $|\cup_{i\in I}S(P_{i})|\le 33$. Hence there do not exist six points $P_{1},P_{2},P_{3},P_{4}, P_{5},P_{6}\in PG_{2}(8)$ such that $PG_{2}(8)\backslash(K\cup\{P_{1},P_{2},P_{3},P_{4},P_{5},P_{6}\})$ contains no line.


\begin{thebibliography}{10}

\bibitem{C2004}
S.~Cioab\v{a}.
\newblock Bounds on the {T}ur\'{a}n density of {$PG_{3}(2)$}.
\newblock {\em Electronic J. Combin.}, 11:1--7, 2004.

\bibitem{DF2000}
D.~de~Caen and Z.~F\"{u}redi.
\newblock The maximum size of 3-uniform hypergraphs not containing a {F}ano
  plane.
\newblock {\em J. Combin. Theory Ser. B}, 78:274--276, 2000.

\bibitem{ES1946}
P.~Erd\H{o}s and A.~H. Stone.
\newblock On the structure of linear graphs.
\newblock {\em Bull. Amer. Math. Soc.}, 52:1087--1091, 1946.

\bibitem{F1991}
Z.~F\"{u}redi.
\newblock Tur\'{a}n type problems.
\newblock In {\em Surveys in combinatorics, 1991 ({G}uildford, 1991)}, volume
  166 of {\em London Math. Soc. Lecture Note Ser.}, pages 253--300. Cambridge
  Univ. Press, Cambridge, 1991.

\bibitem{FPS2003}
Z.~F\"{u}redi, O.~Pikhurko, and M.~Simonovits.
\newblock The {T}ur\'{a}n density of the hypergraph {$\{abc,ade,bde,cde\}$}.
\newblock {\em Electronic J. Combin.}, 10:7pp, 2003.

\bibitem{FPS2005}
Z.~F\"{u}redi, O.~Pikhurko, and M.~Simonovits.
\newblock On triple systems with independent neighbourhoods.
\newblock {\em Combin. Probab. Comput.}, 14(5-6):795, 2005.

\bibitem{FS2005}
Z.~F\"{u}redi and M.~Simonovits.
\newblock Triple systems not containing a fano configuration.
\newblock {\em Combin. Probab. Comput.}, 14:467--484, 2005.

\bibitem{H2005}
J.~W.~P. Hirschfeld.
\newblock {\em Projective {G}eometries over {F}inite {F}ields}.
\newblock Oxford Mathematical Monographs. Oxford Science Publications, second
  edition, 2005.

\bibitem{K2005}
P.~Keevash.
\newblock The {T}ur\'{a}n problem for projective geometries.
\newblock {\em J. Combin. Theory Ser. A}, 111:289--309, 2005.

\bibitem{K11}
P.~Keevash.
\newblock Hypergraph {T}ur\'{a}n problems.
\newblock {\em Surveys in Combinatorics, Cambridge University Press}, 2011,
  83--140.

\bibitem{KS2005}
P.~Keevash and B.~Sudakov.
\newblock The exact {T}ur\'{a}n number of the {F}ano plane.
\newblock {\em Combinatorica}, 25:561--574, 2005.

\bibitem{KeeS2005}
P.~Keevash and B.~Sudakov.
\newblock On a hypergraph {T}ur\'{a}n problem of {F}rankl.
\newblock {\em Combinatorica}, 25(6):673--706, 2005.

\bibitem{MR2002}
D.~Mubayi and V.~R\"{o}dl.
\newblock On the {T}ur\'{a}n {N}umber of {T}riple {S}ystems.
\newblock {\em J. Combin. Theory Ser. A}, 100(1):136--152, 2002.

\bibitem{T1988}
G.~Tallini.
\newblock On blocking sets in finite projective and affine spaces.
\newblock {\em Annals of Discrete Mathematics}, 37, 1988.

\bibitem{T1995}
J.~Thas.
\newblock Projective geometry over a finite field.
\newblock In {\em Handbook of Incidence Geometry, (ed. F. Buekenhout), Chapter
  7, North-Holland}, pages 295--347. 1995.

\end{thebibliography}
\end{document}